\documentclass[11pt,twoside,a4paper]{article}
\usepackage{amsmath,amssymb,amsthm}
\usepackage{graphicx}
\usepackage[english]{babel}
\usepackage{scrlayer-scrpage}
\usepackage[explicit]{titlesec}
\usepackage{libertine}
\usepackage{tikz}
\usetikzlibrary{calc}
\usepackage{tikzpagenodes}
\usepackage{hyperref}
\usepackage[margin=1.5in]{geometry}

\theoremstyle{plain}
\newtheorem{St}{Theorem}[section]
\newtheorem{Le}[St]{Lemma}
\newtheorem{Gev}[St]{Corollary}
\theoremstyle{definition}
\newtheorem{Ex}[St]{Example}
\newtheorem{Def}[St]{Definition}
\newtheorem{Opm}[St]{Remark}

\DeclareMathOperator\pg{\mathrm{PG}}
\DeclareMathOperator\ag{\mathrm{AG}}
\DeclareMathOperator\im{\mathrm{Im}}

\newcommand{\cL}{\mathcal{L}}
\newcommand{\gauss}[2]{\genfrac{[}{]}{0pt}{}{#1}{#2}}

\begin{document}
	\title{Cameron-Liebler $k$-sets in subspaces and non-existence conditions}

 \footnotetext{$*$ Department of Mathematics: Analysis, Logic and Discrete Mathematics, Ghent University, Krijgslaan 281, Building S8, 9000 Gent, Belgium (Email: leo.storme@ugent.be) \newline (http://cage.ugent.be/$\sim$ls).}%
    \footnotetext{$\dagger$ Department of Mathematics and Data Science, University of Brussels (VUB),  Pleinlaan 2, Building G, 1050 Brussels, Belgium  (Email: Jan.De.Beule@vub.ac.be, Jonathan.Mannaert@vub.be)\newline(http://homepages.vub.ac.be/$\sim$jdbeule/,http://homepages.vub.be/$\sim$jonmanna/).}%

\author{J. De Beule $^\dagger$, J. Mannaert$^\dagger$ and L. Storme$^*$}

\maketitle
\begin{abstract}
In this article we generalize the concepts that were used in the PhD thesis of Drudge to classify Cameron-Liebler line classes in PG$(n,q), n\geq 3$, to Cameron-Liebler sets of $k$-spaces in $\pg(n,q)$ and $\ag(n,q)$. In his PhD thesis, Drudge proved that every Cameron-Liebler line class in $\pg(n,q)$ intersects every $3$-dimensional subspace in a Cameron-Liebler line class in that subspace.  We are using the generalization of this result for sets of $k$-spaces in $\pg(n,q)$ and $\ag(n,q)$. Together with a basic counting argument this gives a very strong non-existence condition, $n\geq 3k+3$. This condition can also be improved for $k$-sets in $\ag(n,q)$, with $n\geq 2k+2$.
\end{abstract}

\section{Introduction}

The objective of Cameron and Liebler in \cite{Cameron-Liebler} was to study irreducible collineation groups of $\pg(d,q)$ having equally many point orbits as line orbits. 
Such a group induces a symmetrical tactical decomposition on $\pg(d,q)$. Any line class in such a tactical decomposition is called a Cameron-Liebler line class, and
this is equivalent with the characteristic vector of the line class belonging to the row space of the point-line incidence matrix of $\pg(d,q)$. For 
$d=3$, a Cameron-Liebler line class is characterized by the property that it meets any spread of $\pg(3,q)$ in a fixed number $x$ of lines, where a spread of $\pg(3,q)$
is a partition of the point set in lines. So the parameter $x$ depends only on the size of such a class. Now for $d=3$, it is easy to see that the following line sets are examples:
(1) the empty set, of parameter $x=0$, (2) all lines through a fixed point $p$, of parameter $x=1$, (3) all lines in a fixed plane $\pi$, also of parameter $x=1$, and (4) a  union of (2) and (3) with $p\not\in\pi$, of parameter $x=2$. These examples (and their complements in the set of lines) 
are called {\em trivial examples}, and it was conjectured in \cite{Cameron-Liebler} that no other examples exist. This conjecture
has been disproven by Drudge in \cite{Drudge}, who gave an example in $\pg(3,3)$ of parameter $x=5$; an example that was generalized to an infinite family of parameter
$x=\frac{q^2+1}{2}$ in $\pg(3,q)$, $q$ odd, in \cite{BruenAndDrude}. New non-trivial examples have been discovered by Rodgers in \cite{RodgersPhD,Rodgers2013}, some of them have been generalized to infinite families, see \cite{DeBeule2016,Feng2015,Feng20xx}. Generally spoken, non-trivial examples are rare.  Furthermore,
non-existence results of Cameron-Liebler line classes for particular values of the parameter $x$  have been found, see e.g. \cite{MetschAndGavrilyuk}. In particular, this article contains a result that excludes roughly half of the possible parameters.

 In \cite{DrudgeThesis}, Drudge started to characterize Cameron-Liebler line classes in $\pg(n,q)$ by intersection properties with $3$-dimensional subspaces. 
Drudge showed that the only Cameron-Liebler line classes in $\pg(n,3)$, $n\geq 4$, are trivial, which means indeed of the same structure as the trivial
examples known in $\pg(3,q)$.  We were informed by A. Gavrilyuk that these results have been extended for $\pg(n, 4)$, $n \geq 4$, by Gavrilyuk and
Mogilnykh in \cite{GavMog}, for $\pg(n, 5)$, $n \geq 4$, by Matkin in \cite{Mat18}.


In recent years, Cameron-Liebler line classes were generalized to Cameron-Liebler $k$-sets in $\pg(n,q)$ and $\ag(n,q)$, see \cite{Jozefien, Me1, Me2}. These have been completely classified for $q\leq5$ by Filmus and
Ihringer in \cite{BDF}. Our goal will be to combine these two ideas. We want to use a similar approach of the PhD thesis of Drudge and find out what we can generalize to Cameron-Liebler sets of $k$-spaces in $\pg(n,q)$ or in $\ag(n,q)$, or \emph{Cameron-Liebler $k$-sets} for short. This will prove to be very effective since we are able to prove the following non-existence conditions.
\begin{St}
Suppose that $n\geq 3k+3$ and $k\geq 1$. Let $\mathcal{L}$ be a Cameron-Liebler $k$-set of parameter $x$ in $\pg(n,q)$  such that $\mathcal{L}$ is not a point-pencil, nor the empty set. Then it holds that
$$x\geq\frac{q^{n-k}-1}{q^{2k+2}-1}+1.$$
\end{St}
The previous theorem is also valid for $\ag(n,q)$, but can be improved in that case to the following statement.
\begin{St}
Suppose that $n\geq 2k+2$ and $k\geq 1$. Let $\mathcal{L}$ be a Cameron-Liebler $k$-set of parameter $x$ in $\ag(n,q)$ such that $\mathcal{L}$ is not a point-pencil, nor the empty set. Then it holds that
$$x\geq2\left(\frac{q^{n-k}-1}{q^{k+1}-1}\right)+1.$$
\end{St}

These results are significant improvements to what is known so far in the literature. We should also note that we are only interested in classifying Cameron-Liebler $k$-sets in $\pg(n,q)$ (or $\ag(n,q)$) of parameter $x\leq \frac12\frac{q^{n+1}-1}{q^{k+1}-1}$ (or $x\leq \frac12q^{n-k}$), which is the maximal possible parameter of a Cameron-Liebler $k$-set divided by two. Classifying these parameters would yield a classification of all parameters, since it is known that the complement of a Cameron-Liebler $k$-set of parameter $x$ in $\pg(n,q)$ (or $\ag(n,q)$) is a Cameron-Liebler $k$-set of parameter $\frac{q^{n+1}-1}{q^{k+1}-1}-x$ (or $q^{n-k}-x$ respectively), see \cite[Lemma 3.1]{Jozefien}.
Section 2 will give some useful results that we will use along the way. Section 3 gives the generalizations of the results of the PhD thesis of Drudge. We are also able to improve some of these results in a slight way. Section 4 is based on a counting argument that will be used in Section 6 to obtain the non-existence results, while Section 5 gives some  consequences of the counting argument. We end this article with a small conclusion in Section 7.

\section{Preliminaries}
In this section, we give some results and definitions on Cameron-Liebler $k$-sets in $\ag(n,q)$ and $\pg(n,q)$. These objects are sets of $k$-spaces in $\pg(n,q)$ or $\ag(n,q)$ respectively, with $n\geq 2k+1\geq 3$. We will treat the results for $\pg(n,q)$ and $\ag(n,q)$ separately, but first we give the following notation. 
\begin{Def}
Suppose that $p$ is a point and $\tau$ an $i$-dimensional subspace in $\pg(n,q)$ or $\ag(n,q)$. Then we denote by $[p]_k$  the set of $k$-spaces through $p$, $[\tau]_k$  the set of $k$-spaces in $\tau$ and $[p,\tau]_k:=[p]_k\cap [\tau]_k$.
\end{Def}

\begin{Def}
A {\em $k$-spread} of $\pg(n,q)$, respectively $\ag(n,q)$, is a set of $k$-dimensional subspaces of $\pg(n,q)$, respectively $\ag(n,q)$, partitioning the point set
of $\pg(n,q)$, respectively $\ag(n,q)$.
\end{Def}

\begin{Opm}
Recall the well-known fact that $\pg(n,q)$ has $k$-spreads if and only if $k+1 \mid n+1$, see e.g. \cite[Corollary 3.7]{Hirschfeld}.
\end{Opm}

We end with a useful notation for the number of $(a-1)$-spaces in $\pg(b-1,q)$, with $a,b\in \mathbb{N}$. This number is denoted by
$$\gauss{b}{a}_q=\frac{(q^b-1)\cdots (q^{b-a+1}-1)}{(q^a-1)\cdots (q-1)}.$$
This number is known as the \emph{Gaussian binomial coefficient}.

\subsection{The projective space $\pg(n,q)$}

Let $P_n$ be the $0,1$ valued matrix, rows indexed by the points and columns by the $k$-dimensional subspaces ($k$-spaces) and such that $P_n$ at position $(p,\pi)$
equals $1$ if and only if $p \in \pi$ and $0$ otherwise. This matrix is called the point-$k$-space incidence matrix of $\pg(n,q)$, and is used to define Cameron-Liebler $k$-sets. 


\begin{Def}
A \emph{Cameron-Liebler set of $k$-spaces in $\pg(n,q)$} is a set $\mathcal{L}$ of $k$-spaces in $\pg(n,q)$ such that its characteristic vector $\chi_\mathcal{L}$
is a vector in the row space of the matrix $P_n$ (equivalently, in the column space of $P_n^T$, denoted $\im(P_n^T)$). A Cameron-Liebler set of $k$-spaces in $\pg(n,q)$ has 
{\em parameter $x$} if and only if 
\[
|\mathcal{L}|=x \gauss{n}{k}_q\,.
\]
\end{Def}

\begin{Ex}\label{ex:trivial_projective}
The following examples are {\em trivial examples}: (1) the empty set, of parameter $x=0$, (2) all $k$-spaces through a fixed point $p$, 
of parameter $x=1$, (3) all $k$-spaces in a fixed hyperplane $\pi$, of parameter $x=\frac{q^{n-k}-1}{q^{k+1}-1}$, (4)
a  union of (2) and (3) with $p\not\in\pi$, and (5) all complements of the examples above. Note that $\frac{q^{n-k}-1}{q^{k+1}-1}$ is an integer if and only if $(k+1)\mid (n+1)$. 
\end{Ex}

Besides these trivial examples, no example is known for $k\geq 2$. 
Yet Cameron-Liebler $k$-sets have many equivalent definitions to work with. The following theorem summarizes them.

\begin{St}\cite[Theorem 2.2]{Jozefien}\label{EquivalenceProj}
	Let \(\mathcal{L}\) be a non-empty set of $k$-spaces in $\pg(n,q)$, $n \geq 2k+1$, with characteristic vector $\chi_\mathcal{L}$, and $x$ so that 
	$|\mathcal{L}|= x\gauss{n}{k}_q$. Then the following properties are equivalent.
	\begin{enumerate}
		\item $\chi_\mathcal{L} \in \im(P_n^T)=(\ker(P_n))^\perp$, with $P_n$ the point-($k$-space) incidence matrix of $\pg(n,q)$.
		\item For every $k$-space $K$, the number of elements of $\mathcal{L}$ disjoint from $K$ is equal to $(x-\chi_\mathcal{L}(K))\gauss{n-k-1}{k}_q q^{k^2+k}$.
		\item For an $i \in \{1,...,k+1\}$ and a given $k$-space $K$, the number of elements of $\mathcal{L}$, meeting $K$ in a $(k-i)$-space is given by:
		$$\left\{\begin{array}{ll}
			  \left( (x-1) \frac{q^{k+1}-1}{q^{k-i+1}-1}+ q^i \frac{q^{n-k}-1}{q^i-1}\right) q^{i(i-1)} \gauss{n-k-1}{i-1}_q \gauss{k}{i}_q  & \text{ if } K \in \mathcal{L} \\
			xq^{i(i-1)} \gauss{n-k-1}{i-1}_q \gauss{k+1}{i}_q & \text{ if } K \not\in \mathcal{L}.
		\end{array} \right.$$
		\item If $k+1 \mid n+1$, i.e. if and only if $\pg(n,q)$ has $k$-spreads, then $|\mathcal{L}\cap \mathcal{S}|=x$ for any $k$-spread $\mathcal{S}$. 
		(from which it follows that $x \in \mathbb{N}$).
	\end{enumerate}
\end{St}
It is known that the size of $\cL$ can often be deduced from some definitions. We give the following corollary.
\begin{Gev}\label{EquivalencenoParamP}
Suppose that $\cL$ is a set of $k$-spaces in $\pg(n,q)$, $n\geq 2k+1$, and $x$ a fixed number, then the following statements are equivalent.
\begin{enumerate}
\item $\cL$ is a Cameron-Liebler $k$-set of parameter $x$.
\item For every $k$-space $K$, the number of elements of $\mathcal{L}$ disjoint from $K$ is equal to $(x-\chi_\mathcal{L}(K))\gauss{n-k-1}{k}_q q^{k^2+k}$.
\end{enumerate}
\end{Gev}
\begin{proof}
The only fact that we need to prove is that from statement (2) we have that $\cL$ is a Cameron-Liebler $k$-set of parameter $x$. To obtain that parameter, we can simply count the pairs of $k$-spaces $(K_1,K_2)$,  such that $K_1\in \cL$, $K_2\not \in \cL$, and the intersection is the empty set. Since, (see \cite[Section 170]{Segre}), the number of $k$-spaces inside $\pg(n,q)$ disjoint from a fixed $k$-space is known to be
$$q^{(k+1)^2}\gauss{n-k}{k+1}_q,$$
we find that the number of pairs equals
$$|\cL|\left( q^{(k+1)^2}\gauss{n-k}{k+1}_q-(x-1)q^{k^2+k}\gauss{n-k-1}{k}_q\right),$$
and is also equal to
$$\left(\gauss{n+1}{k+1}_q-|\cL|\right) xq^{k^2+k}\gauss{n-k-1}{k}_q.$$
Working out this equality, we obtain that $|\cL|=x\gauss{n}{k}_q$. Hence, $\cL$ has parameter $x$.
\end{proof}

The following non-existence results will be useful later on.

\begin{Le}\cite[Lemma 3.1]{Jozefien}\label{Prop}
Consider a Cameron-Liebler $k$-set of parameter $x$ in $\pg(n,q)$, then $0 \leq x \leq \frac{q^{n+1}-1}{q^{k+1}-1}$.
\end{Le}
The following result gives a non-existence condition for Cameron-Liebler $k$-sets in $\pg(n,q)$ of parameter $x<1$.

\begin{St}\cite[Theorem 4.3]{Jozefien}\label{Non-existence01}
There do not exist Cameron-Liebler $k$-sets in $\pg(n,q)$ of parameter $x \in ]0,1[$ and if $n\geq 3k+2$, then there are no Cameron-Liebler $k$-sets of parameter $x\in]1,2[$.
\end{St}

The following theorem classifies Cameron-Liebler $k$-sets of parameter $x=1$. 
\begin{St}\cite[Theorem 4.1]{Jozefien} \label{CharX=1PG(n,q)}\label{Uniquenessx=1}
Let $\mathcal{L}$ be a Cameron-Liebler $k$-set  with parameter $x = 1$ in $\pg(n,q)$, $n \geq 2k+1$. Then $\mathcal{L}$ consists out of all the $k$-spaces through a fixed point or $n = 2k+1$ and $\mathcal{L}$ is the set of all the $k$-spaces in a hyperplane of PG($2k + 1,q$).
\end{St}

\subsection{The affine space $\ag(n,q)$}
Cameron-Liebler line classes in $\ag(n,q)$ were first described and studied in \cite{Me1} for $n=3$, and generalized for $n \geq 3$ in \cite{Me2}. We thus refer to these articles for this section. Let $A_n$ be the point-($k$-space) incidence matrix in $\ag(n,q)$. Cameron-Liebler $k$-sets in $\ag(n,q)$
are now defined similarly as in the projective case.

\begin{Def}
A \emph{Cameron-Liebler set of $k$-spaces in $\ag(n,q)$} is a set $\mathcal{L}$ of $k$-spaces in $\ag(n,q)$ such that its characteristic vector $\chi_\mathcal{L}$
is a vector in the row space of the matrix $A_n$ (equivalently, in the column space of $A_n^T$, denoted $\im(A_n^T)$). A Cameron-Liebler set of $k$-spaces in $\ag(n,q)$ has 
{\em parameter $x$} if and only if 
\[
|\mathcal{L}|=x \gauss{n}{k}_q\,.
\]
\end{Def}

Note that it has been proven in \cite{Me2} that this definition is equivalent with the following definition using $k$-spreads of $\ag(n,q)$.

\begin{Def}\label{definition_2.12}
A Cameron-Liebler $k$-set in $\ag(n,q)$ is a set of $k$-spaces $\cL$, such that for every $k$-spread $\mathcal{S}$, it holds that $|\mathcal{L}\cap \mathcal{S}|=x$. 
\end{Def}

Note that contrary to the projective space, the affine space $\ag(n,q)$ has $k$-spreads for all $k < n$, illustrated by the following easy example.
\begin{Ex}
Consider $\ag(n,q)$ and its projective closure $\pg(n,q)$, with $\pi_\infty$ the hyperplane at infinity. If we choose $I$ to be a $(k-1)$-space inside $\pi_\infty$, then all the (affine) $k$-spaces through $I$ form a $k$-spread of AG$(n,q)$.
\end{Ex}
\begin{Opm}
Different from $\pg(n,q)$, it holds that the parameter $x$ of a Cameron-Liebler $k$-set in $\ag(n,q)$ is always an integer. This fact follows from the last definition.
\end{Opm}
For the case of lines, we also have the following interesting result. This result is slightly improved in a similar way as was done for Corollary \ref{EquivalencenoParamP}.
\begin{St}\cite[Theorem 3.7]{Me2}\label{EquivalenceAff}
Suppose that $\cL$ is a set of lines in $\ag(n,q)$, with characteristic vector $\chi_\cL$, then the following statements are equivalent:
\begin{enumerate}
\item The set $\cL$ is a Cameron-Liebler line class of parameter $x$.
\item For every line $\ell$, there are $(q^2\frac{q^{n-2}-1}{q-1}+1)(x-\chi_\mathcal{L}(\ell))$ lines of $\cL$ disjoint from $\ell$ and through every point at infinity there are $x$ affine lines of $\cL$.
\end{enumerate}
\end{St}
\begin{proof}
Due to Theorem \ref{EquivalenceAff}, we only need to prove that from (2) to (1) it holds that $|\cL|=x\gauss{n}{1}_q$ and thus the parameter is indeed $x$. This proof is completely similar as for Corollary \ref{EquivalencenoParamP}. Count the pairs $(\ell_1, \ell_2)$ such that $\ell_1\in\cL$, $\ell_2\not \in \cL$, and both are disjoint in $\ag(n,q)$. The number of lines in $\ag(n,q)$ disjoint to a fixed line is equal to $q^4\gauss{n-1}{2}_q-q^2\gauss{n-1}{2}_q+q^{n-1}-1$, which equals the number of lines disjoint in the projective closure, minus those at infinity and adding those affine lines through the unique point at infinity. So  this number equals
$$|\cL|\left( q^4\gauss{n-1}{2}_q-q^2\gauss{n-1}{2}_q+q^{n-1}-1-(x-1)(q^2\frac{q^{n-2}-1}{q-1}+1)\right),$$
and is also equal to
$$\left( \gauss{n+1}{2}_q-\gauss{n}{2}_q-|\cL|\right) x(q^2\frac{q^{n-2}-1}{q-1}+1).$$
Working out this equality, we obtain that $|\cL|=x\gauss{n}{1}_q$. Thus the parameter equals $x$.
\end{proof}

We now give some trivial examples of such Cameron-Liebler $k$-sets. One should compare this example with the given examples in $\pg(n,q)$.

\begin{Ex} The following examples and their complements are known as \emph{trivial examples}: (1) the empty set is a Cameron-Liebler $k$-set, of parameter $x=0$, (2) the set of 
all $k$-spaces through a fixed point $p$, of parameter $x=1$.
\end{Ex}

Note that the set of $k$-spaces contained in a hyperplane of $\ag(n,q)$ is not a Cameron-Liebler $k$-set of $\ag(n,q)$, which is different from the projective case; compare with Example~\ref{ex:trivial_projective}.

Finally, we recall the following theorem from  \cite{Me2}, connecting Cameron-Liebler $k$-sets in $\ag(n,q)$ and $\pg(n,q)$.

\begin{St}\cite[Theorem 1.2]{Me2}\label{ToInftyAndBeond}
Every Cameron-Liebler $k$-set in $\ag(n,q)$ is a Cameron-Liebler $k$-set in its projective closure $\pg(n,q)$, which has the same parameter $x$.
\end{St}
This enables us to translate results in $\pg(n,q)$ to $\ag(n,q)$ and will be of great use. We also have a converse result.
\begin{St}\cite[Theorem 1.1]{Me2}\label{ToInftyAndBeondConv}
Let $\cL$ be a Cameron-Liebler $k$-set with parameter $x$ in PG($n,q$) which does not contain $k$-spaces in some hyperplane $H$. 
Then $\cL$ is a Cameron-Liebler $k$-set with parameter $x$ of $\ag(n, q) \cong \pg(n, q) \setminus H$.
\end{St}

These theorems make sure that all results in $\ag(n,q)$ are at least as strong as in $\pg(n,q)$.
\begin{St}\cite[Theorem 6.5]{Me2}\label{Unique1AG}
Suppose that $\cL$ is a Cameron-Liebler  $k$-set of parameter $x=1$ in $\ag(n,q)$ for $n\geq 2k+1$, then $\cL$ consists of all the $k$-spaces through a fixed point.
\end{St}

\section{Cameron-Liebler sets in subspaces}
In this section, we will focus on the behaviour of Cameron-Liebler sets in subspaces of $\pg(n,q)$ and $\ag(n,q)$. We will treat both cases independently. This was already observed for Cameron-Liebler line classes in $\pg(n,q)$ in \cite{DrudgeThesis}. We will focus on their results and we will generalize these in the following subsections.

\subsection{The projective case}

The following theorem is in fact a generalization of Theorem \ref{DrudgePointPencil}. In his PhD thesis, Drudge was interested in the behaviour of Cameron-Liebler line classes in subspaces. It was proven that every Cameron-Liebler line class in $\pg(n,q)$ restricted to a $3$-dimensional subspace is a Cameron-Liebler line class in that subspace.  The generalization of this particular result was already briefly discussed in \cite{BDF}. Since it was not written down in a lemma or theorem, we formally write it down in the following theorem.

\begin{St}[Folklore]\label{Subsp}
Consider a Cameron-Liebler $k$-space $\mathcal{L}$ in $\pg(n,q)$, with $n \geq 2k+1$, 
and let $\pi$ be a subspace of dimension $i\geq k+1$. Then $\mathcal{L}\cap [\pi]_k$  is also a Cameron-Liebler $k$-space in $\pi$.
\end{St}
\begin{proof}
Let $\pi$ be a subspace of dimension $i\geq k+1$. After a suitable reordering of the points and $k$-spaces, 
it is clear that the point-($k$-space) incidence matrix $P_n$ of $\pg(n,q)$  is equal to 

\begin{equation}\label{IncidenceMatrixPn}
P_n= \begin{pmatrix}
P_\pi & B_2  \\
\bar{0} & B_3
\end{pmatrix},
\end{equation}
where $P_\pi$ is the incidence matrix of $\pi$.

Since $\mathcal{L}$ is a Cameron-Liebler $k$-set in $\pg(n,q)$, the characteristic vector $\chi_\mathcal{L} \in \im(P_n^T)$. 
So there exists a vector $\begin{pmatrix} v_1 \\ v_2 \end{pmatrix}$ such that 
$$\chi_{\mathcal{L}}=  \begin{pmatrix}
P_\pi^T & \bar{0}^T \\
B_2^T & B_3^T
\end{pmatrix} \cdot \begin{pmatrix} v_1 \\ v_2 \end{pmatrix}= \begin{pmatrix} P_\pi^T \cdot v_1 \\ B_2^T \cdot v_1 +B_3^T \cdot v_2 \end{pmatrix}.$$
Hence, we can conclude that $\chi_{\mathcal{L}}$ restricted to $\pi$ is equal to $P_\pi^T \cdot v_1$. Thus $\chi_{\mathcal{L}_{|\pi}} \in $ Im($P_\pi^T$), 
which proves the theorem.
\end{proof}
As indicated in the introduction, Drudge (\cite{DrudgeThesis}) used intersection properties with $3$-spaces to 
characterize Cameron-Liebler line classes in $\pg(n,q)$.  A first step in this approach is the following theorem.

\begin{St}\cite[Theorem 6.1]{DrudgeThesis}\label{DrudgePointPencil}
Suppose that $\cL$ is a Cameron-Liebler line class in $\pg(n,q)$, such that there exists a $3$-dimensional 
subspace $\pi$ for which $\cL\cap [\pi]_1$ is a point-pencil. Then $\cL$ is a point-pencil in $\pg(n,q)$.
\end{St}

The proof given in \cite{DrudgeThesis} is based on a statement that has later been generalized to the following lemma. 
We use a similar technique to obtain the $k$-space analogon.

\begin{Le}\cite[Lemma 2.12]{Jozefien}\label{DrudgeArgum}
Let $\mathcal{L}$ be a Cameron-Liebler $k$-set in $\pg(n,q)$, then we find the 
following equality for every point $p$ and every $i$-dimensional subspace $\tau$, 
with $p\in \tau$ and $i\geq k+1$,
\[
\left| [p]_k\cap \mathcal{L} \right|+ \frac{\gauss{n-1}{k}_q (q^k-1)}{\gauss{i-1}{k}_q (q^i-1)} \left| [\tau]_k\cap \mathcal{L} \right| 
= \frac{\gauss{n-1}{k}_q }{\gauss{i-1}{k}_q } \left| [p,\tau]_k\cap \mathcal{L} \right| + \frac{q^k-1}{q^n-1} \left| \mathcal{L} \right|\,.
\]
\end{Le}

\begin{St} \label{DrudgeGeneralPG}
Let $n\geq 2k+1$. Suppose that $\mathcal{L}$ is a Cameron-Liebler $k$-set in $\pg(n,q)$, 
such that there exists an $i$-space $\pi$, with $i\geq k+1$, and such that $\mathcal{L}\cap [\pi]_k$ 
consists of the set of $k$-spaces through a point $p \in \pi$. Then $\mathcal{L}$ is the  set of $k$-spaces 
through this same point $p$. 
\end{St}
\begin{proof}
We will replicate the argument used in \cite[Theorem 6.1]{DrudgeThesis}. Consider the point $p$ and the $i$-dimensional subspace $\pi$, hence $p\in \pi$. 
From the assumptions on $\mathcal{L}$, 
\[
|[\pi]_k\cap \mathcal{L}|=|[p,\pi]_k\cap \mathcal{L}|=\gauss{i}{k}_q\,.
\]
Now using Lemma \ref{DrudgeArgum}, we obtain
$$|[p]_k\cap \mathcal{L}|+ \frac{\gauss{n-1}{k}_q (q^k-1)}{\gauss{i-1}{k}_q (q^i-1)} \gauss{i}{k}_q=\frac{\gauss{n-1}{k}_q }{\gauss{i-1}{k}_q } \gauss{i}{k}_q+\frac{q^k-1}{q^n-1}x\gauss{n}{k}_q.$$
We observe that in general, $|[p]_k\cap \mathcal{L}|\leq \gauss{n}{k}_q$, so the equation above yields
$$\gauss{n}{k}_q+ \frac{\gauss{n-1}{k}_q (q^k-1)}{\gauss{i-1}{k}_q (q^i-1)} \gauss{i}{k}_q\geq\frac{\gauss{n-1}{k}_q }{\gauss{i-1}{k}_q } \gauss{i}{k}_q+\frac{q^k-1}{q^n-1}x\gauss{n}{k}_q.$$
Hence,
$$(q^n-1)+\frac{(q^{n-k}-1)(q^k-1)}{q^{i-k}-1}-\frac{(q^{n-k}-1)(q^i-1)}{q^{i-k}-1}\geq (q^k-1)x.$$
Further calculations give that $x\leq 1$. So this fact, together with Theorem \ref{Non-existence01}, Theorem \ref{Uniquenessx=1} and the definition of $\mathcal{L}\cap[\pi]_k$, we find that $x=1$ and that $\mathcal{L}=[p]_k$.
\end{proof}
\begin{Opm}
The previous theorem is in fact an improvement of the result of Drudge for the case that $k=1$. If one should compare this theorem with Theorem \ref{DrudgePointPencil}, one can find that the case $i=2$ was not yet covered.
\end{Opm}

\subsection{The affine case}
As mentioned in the preliminaries, some results for Cameron-Liebler $k$-sets in $\pg(n,q)$ 
can be translated to results in $\ag(n,q)$. We will do this for Theorem \ref{Subsp}, hence we 
obtain the following result. We should note that this argument is similar to that of Theorem 3.8 in \cite{Me2}. 
But to make our article self contained, we decided to repeat it.
\begin{St}\label{th:affine-induced}
Let $\mathcal{L}$ be a Cameron-Liebler $k$-set in $\ag(n,q)$. Suppose now that $\pi$ is an 
$i$-dimensional subspace in $\ag(n,q)$, with $i\geq k+1$, then $\mathcal{L}\cap [\pi]_k$ 
is a Cameron-Liebler $k$-set in $\pi$.
\end{St}
\begin{proof}
Consider the projective closure $\pg(n,q)$ of $\ag(n,q)$ and let $\pi_\infty$ be the 
hyperplane at infinity.  Let $\pi$ be an $i$-dimensional space in $\pg(n,q)$, hence 
$\dim(\pi\cap\pi_\infty)\geq k$. For a projective subspace $\pi'$, we will denote
its affine part as $\pi'_A$.

We will prove that $\mathcal{L}$ restricted to $\pi_{A}$ is a Cameron-Liebler $k$-set. 
Choose a $(k-1)$-space $I$ in $\pi\cap \pi_\infty$ and denote $E$ as the set of affine 
$k$-spaces containing $I$ at infinity and not contained in $\pi_A$. Then every $k$-space $K\in E$ is 
disjoint with every $k$-space contained in $\pi_A$, and all $k$-spaces in $E$ are parallel. Furthermore, 
the set of affine $k$-spaces in $E$ cover the affine points not in $\pi_A$. Hence, any affine $k$-spread 
$\mathcal{S}$ in $\pi_A$ can be extended to a $k$-spread $\mathcal{S}':=\mathcal{S}\cup E.$
Consequently,
$$x=|\mathcal{L}\cap\mathcal{S}'|=|\mathcal{L}\cap \mathcal{S}|+|\mathcal{L}\cap E|.$$
Since $|\mathcal{L}\cap E|$ is fixed and $\mathcal{S}$ is an arbitrary $k$-spread of $\pi_A$, $|\mathcal{L}\cap \mathcal{S}|$ is a constant for all $k$-spreads of $\pi_A$. The theorem follows by Definition \ref{definition_2.12}.
\end{proof}

The following theorem translates Theorem \ref{DrudgePointPencil} to the affine case.
\begin{St}\label{PPAG}
Let $n\geq 2k+1$ and $i\geq k+1$. Suppose that $\mathcal{L}$ is a Cameron-Liebler $k$-set in $\ag(n,q)$, such that there exists an $i$-space $\pi$ in $\ag(n,q)$, with $\mathcal{L}\cap[\pi]_k$ a set of  $k$-spaces through a fixed point $p$. Then $\mathcal{L}$ is a Cameron-Liebler $k$-set of parameter $x=1$ and, hence, the set of $k$-spaces through $p$.
\end{St}
\begin{proof}
Suppose that $\mathcal{L}$ is a Cameron-Liebler $k$-set in $\ag(n,q)$ such that there exists an $i$-space $\pi$ 
with the property of the theorem. Then, from Theorem~\ref{ToInftyAndBeond}, it follows that $\mathcal{L}$ 
is a Cameron-Liebler $k$-set in $\pg(n,q)$ with the same property. Hence, using 
Theorem~\ref{DrudgeGeneralPG}, we now obtain that $\mathcal{L}$ is a point-pencil in $\pg(n,q)$. 
Due to the size of $\mathcal{L}$, this point-pencil is defined by  an affine point.
\end{proof}

\section{A property of Cameron-Liebler $k$-sets}

In this section, we will focus on the connection between the parameter of a Cameron-Liebler $k$-set
and the parameters of the induced Cameron-Liebler $k$-sets in a subspace, in $\pg(n,q)$ and $\ag(n,q)$.

\subsection{Cameron-Liebler $k$-sets in $\pg(n,q)$}


\begin{Le}\label{Aid}
Suppose that $\mathcal{L}$ is a Cameron-Liebler $k$-set in $\pg(n,q)$ of parameter $x$. 
Then for every  $t$, such that  $2k+1 \leq t\leq n-1$ and $n\geq 2k+2$, and an arbitrary $k$-space $K$ in $\pg(n,q)$,
\begin{equation}\label{Thex}
x=\frac{\left(\sum_{K \in \pi_i}x_{\pi_i}-  \gauss{n-k}{t-k}_q\chi_\mathcal{L}(K) \right)} {\gauss{n-k-1}{t-k-1}_q }+\chi_\mathcal{L}(K)\,,
\end{equation}
where  $\chi_{\mathcal{L}}(K)$ is the value of the characteristic vector of $\mathcal{L}$ at position $K$,
$x_{\pi_i}$ the parameter of the Cameron-Liebler $k$-set induced in $\pi_i$, 
and where the sum runs over all $t$-spaces $\pi_i$ through $K$.
\end{Le}
\begin{proof}
Fix a $k$-space $K$ and choose $t$ such that $2k+1 \leq t \leq n-1$. Let $R$ be the set of elements of $\mathcal{L}$ disjoint to $K$, 
and $S$ the set of $t$-spaces through $K$. Count the number $n$ of pairs $(x,y)$, $x \in R$ and $y \in S$, 
for which $x \subset y$. First choose a fixed $\pi \in S$. By Theorem \ref{Subsp}, 
$\mathcal{L}_{\pi} := \mathcal{L}\cap [\pi]_k$ is a Cameron-Liebler $k$-set in $\pi$. By 
Theorem \ref{EquivalenceProj} (2), $\mathcal{L}_{\pi}$ is characterized by the property that for every $k$-space $K \subset \pi$, 
the number of $k$-spaces of $\mathcal{L}_{\pi}$ skew to $K$ equals 
\[
(x_\pi-\chi_{\mathcal{L}_{\pi}}(K))q^{k^2+k}\begin{bmatrix}t-k-1 \\ k \end{bmatrix}_q\,.
\]
It is clear that $|S| = \gauss{n-k}{t-k}_q$, and since only $x_{\pi}$ depends on $\pi$, we find, using that $\chi_{\mathcal{L}_{\pi}}(K)=\chi_{\mathcal{L}}(K)$,
\begin{equation}\label{eq1}
\begin{split}
n &=\sum_{K \in \pi_i}\left(x_{\pi_i}-\chi_{\mathcal{L}_{\pi_i}}(K)\right)q^{k^2+k}\begin{bmatrix}t-k-1 \\ k \end{bmatrix}_q\\
&= \left(\sum_{K \in \pi_i}x_{\pi_i}-\begin{bmatrix}n-k \\ t-k \end{bmatrix}_q\chi_{\mathcal{L}}(K)\right)q^{k^2+k}\begin{bmatrix}t-k-1 \\ k \end{bmatrix}_q.
\end{split}
\end{equation}
Secondly, choose any $k$-space in $R$. Since two disjoint $k$-spaces of $\pg(n,q)$ are contained in exactly $\gauss{n-2k-1}{t-2k-1}_q$ $t$-spaces,
we conclude that
\begin{equation}\label{eq2}
|R| = \left(\sum_{K \in \pi_i}x_{\pi_i}-\gauss{n-k}{t-k}_q\chi_{\mathcal{L}}(K)\right)q^{k^2+k}\frac{\gauss{t-k-1}{k}_q}{\gauss{n-2k-1}{t-2k-1}_q}.
\end{equation}
By Theorem~\ref{EquivalenceProj}, (2), we now find 
\begin{equation}\left(\sum_{K \in \pi_i}x_{\pi_i}-\gauss{n-k}{t-k}_q\chi_{\mathcal{L}}(K)\right)\frac{\gauss{t-k-1}{k}_q}{\gauss{n-2k-1}{t-2k-1}_q}q^{k^2+k} = q^{k^2+k}(x-\chi_{\mathcal{L}}(K))
\gauss{n-k-1}{k}_q. 
\end{equation}
This equality can be simplified to 
\begin{equation}\sum_{K \in \pi_i}x_{\pi_i}-\gauss{n-k}{t-k}_q\chi_{\mathcal{L}}(K) = (x-\chi_{\mathcal{L}}(K))
\gauss{n-k-1}{t-k-1}_q\,,
\end{equation}
from which \eqref{Thex} now follows.
\end{proof}

\subsection{Cameron-Liebler line classes in $\ag(n,q)$}

Lemma~\ref{Aid} can be translated to Cameron-Liebler line classes in the affine space. The lack of a generalization of Theorem \ref{EquivalenceAff} 
seems to cause that this translation is only possible for line classes and not for Cameron-Liebler $k$-sets.

\begin{Le}\label{AidAff}
Suppose that $\mathcal{L}$ is a Cameron-Liebler line class in $\ag(n,q)$, with $n\geq 4$. Then for every $t$ such that $3\leq t\leq n-1$ and 
for an arbitrary line $\ell$ in $\ag(n,q)$, 
\[
x=\frac{q^{t-2}-1}{q^{n-2}-1}\frac{\sum_{\ell \in \pi_i} x_{\pi_i}}{\gauss{n-3}{t-3}_q}-\frac{q^{n-1}-1}{q^{t-1}-1}\chi_\mathcal{L}(\ell)+\chi_\mathcal{L}(\ell),
\]
where  $\chi_{\mathcal{L}}(\ell)$ is the value of the characteristic vector of $\mathcal{L}$ at position $\ell$, $x_{\pi_i}$ the parameter of the Cameron-Liebler 
line class induced in $\pi_i$, and  where the sum runs over all $t$-spaces through $\ell$.
\end{Le}
\begin{proof}
Fix a line $\ell$ in $\ag(n,q)$ and choose $t$ such that $3\leq t\leq n-1$. We consider the projective completion of $\ag(n,q)$. 
Let $R$ be the set of projective lines of $\mathcal{L}$ (i.e. the projective completion of the lines in $\mathcal{L}$) disjoint to $\ell$, and $S$
the set of $t$-spaces through $\ell$. Count the number $n$ of pairs $(x,y)$, $x \in R$ and $y \in S$, for which $x \subset y$. First choose a fixed 
$\pi \in S$. By Theorem~\ref{th:affine-induced}, $\mathcal{L}_{\pi} := \mathcal{L}\cap [\pi]_q$ is a Cameron-Liebler line class in $\pi$. 
By Theorem~\ref{EquivalenceAff} (2), $\mathcal{L}_{\pi}$ is characterized by the property that for every affine line $\ell \subset \pi$, the number of lines 
of $\mathcal{L}_{\pi}$ disjoint from $\ell$ equals 
\[
(q^2\frac{q^{t-2}-1}{q-1}+1)(x_{\pi}-\chi_\mathcal{L}(\ell))\,,
\]
and where precisely $x_{\pi}-\chi_\mathcal{L}(\ell)$ of these lines are parallel with $\ell$. It is clear that $|S| = \gauss{n-1}{t-1}_q$,
and since only $x_{\pi}$ depends on $\pi$, we find, using that $\chi_\mathcal{L_{\pi}}(\ell) = \chi_\mathcal{L}(\ell)$,
\begin{equation}
\begin{split}
n &= \sum_{\ell\in \pi_i} q^2\frac{q^{t-2}-1}{q-1}(x_{\pi_i}-\chi_{\mathcal{L}_{\pi_i}}(\ell))\\
&=q^2\frac{q^{t-2}-1}{q-1}\left( \sum_{\ell\in \pi_i}x_{\pi_i}-\gauss{n-1}{t-1}_q \chi_\mathcal{L}(\ell) \right)\,.
\end{split}
\end{equation}

Secondly, choose any line in $R$. Since two disjoint lines in $\pg(n,q)$ are contained in exactly $\gauss{n-3}{t-3}_q$ $t$-spaces, we conclude that
\begin{equation*}
\begin{split}
|R|&=\frac{q^2}{\gauss{n-3}{t-3}_q}\frac{q^{t-2}-1}{q-1}\left( \sum_{\ell\in \pi_i}x_{\pi_i}-\gauss{n-1}{t-1}_q \chi_\mathcal{L}(\ell) \right).
\end{split}
\end{equation*}
By Theorem~\ref{EquivalenceAff} (2), we now find that 
\[
\frac{q^2}{\gauss{n-3}{t-3}_q}\frac{q^{t-2}-1}{q-1}\left( \sum_{\ell\in \pi_i}x_{\pi_i}-\gauss{n-1}{t-1}_q \chi_\mathcal{L}(\ell) \right) = q^2\frac{q^{n-2}-1}{q-1}(x-\chi_\mathcal{L}(\ell))\,.
\]
This equality can be simplified to 
\[
x=\frac{q^{t-2}-1}{q^{n-2}-1}\frac{\sum_{\ell \in \pi_i} x_{\pi_i}}{\gauss{n-3}{t-3}_q}-\frac{q^{n-1}-1}{q^{t-1}-1}\chi_\mathcal{L}(\ell)+\chi_\mathcal{L}(\ell)\,,
\]
which is the statement of the lemma.
\end{proof}

\section{Elaborations obtained from the counting argument}

In this section, we discuss some  results that can be shown using the results from the previous sections. 

\subsection{Non-integer parameters}
We have mentioned that, unlike $\ag(n,q)$, in $\pg(n,q)$ there exist Cameron-Liebler $k$-sets with a parameter that is not an integer. An example of such a class is given in Example~\ref{ex:trivial_projective}. The parameter of a Cameron-Liebler $k$-set in $\pg(n,q)$ can only be a non-integer if and only if $(k+1)\nmid (n+1)$, since, if $(k+1)\mid (n+1)$, then $k$-spreads in $\pg(n,q)$ exist, 
and due to Theorem \ref{EquivalenceProj} (4) this requires that $x$ is an integer. \\

If $(k+1)\nmid (n+1)$, the number of possible parameters increases greatly. Our goal is to limit this increase significantly. This is done in  the following theorem.


\begin{St}\label{non-integer}
Suppose that $\mathcal{L}$ is a non-empty Cameron-Liebler $k$-set in $\pg(n,q)$ of parameter $x$. 
Then we have that for every  $t$, such that $(k+1)\mid (t+1)$ and $2k+1 \leq t\leq n-1$, that
$$x=1+\frac{C}{\gauss{n-k-1}{t-k-1}_q},$$
for some $C \in \mathbb{N}$.
\end{St}
\begin{proof}
Recall Lemma \ref{Aid}, by which for an arbitrary $k$-space $K \in \mathcal{L}$, it holds that
\begin{equation}\label{HelpLater}x=\chi_\mathcal{L}(K)+\frac{\left(\sum_{K \in \pi_i}x_{\pi_i}-  \gauss{n-k}{t-k}_q\chi_\mathcal{L}(K) \right)} {\gauss{n-k-1}{t-k-1}_q }.\end{equation}
Hence, if we choose $K \in \mathcal{L}$, we obtain that $x=1+\frac{\left(\sum_{K \in \pi_i}x_{\pi_i}-  \gauss{n-k}{t-k}_q \right)} {\gauss{n-k-1}{t-k-1}_q }.$
Remark that due to the fact that $(k+1)\mid (t+1)$ and $K \in \mathcal{L}$, we know that every $x_{\pi_i}\in \mathbb{N}\setminus \{0\}$. Hence, $\sum_{K \in \pi_i}x_{\pi_i}\geq \gauss{n-k}{t-k}_q$ and thus $\sum_{K \in \pi_i}x_{\pi_i}-  \gauss{ n-k}{t-k}_q \in \mathbb{N}$. 

It follows that $$x=1+\frac{C}{\gauss{n-k-1}{t-k-1}_q},$$
for $C \in \mathbb{N}$. 
\end{proof}
This theorem proves how the decimal part of a parameter $x$ can look like. It means that due to the fact that $C\in \mathbb{N}$, there are only a finite number of decimal parts allowed. Hence, for $C\in \{1,\ldots ,\gauss{n-k-1}{t-k-1}_q\}$, to get this number of decimal parts as small as possible, one can choose $t=2k+1$, if $n\geq 2k+1$.
\begin{Opm}
Let us remark that the method of Theorem \ref{non-integer} also proves the first part of Theorem \ref{Non-existence01}, which is a non-existence result for Cameron-Liebler $k$-sets of parameter $x \in]0,1[$, with $n \geq 2k+2$. 
\end{Opm}

\begin{Opm}
Note that similar theorems for $\ag(n,q)$ are not interesting since every parameter $x$ is an integer because $k$-spreads exist in every affine space $\ag(n,q)$.
\end{Opm}

\subsection{Upper bound on the number of parameters in $\pg(n,q)$}
In the previous subsection, we have proven that the number of possible decimal parts of a parameter $x\not\in \mathbb{N}$ is finite. This allows us to determine an upper bound on the possible parameters $x$. This will be done using Lemma \ref{Non-existence01}, since this lemma shows that $$0\leq x\leq \frac{q^{n+1}-1}{q^{k+1}-1}.$$
If $(k+1) \mid (n+1)$ and $x$ is an integer, the number of possible parameters is in general all integers in this interval, thus we have an upper bound  equal to $\frac{q^{n+1}-1}{q^{k+1}-1}+1$. But when non-integers are allowed, this method can not be used, and at first sight there are an infinite number of possible parameters. Here we use the previous section to slim down this set to a finite set. We will be using Theorem \ref{non-integer} to estimate the maximal number of possible parameters $x$ in these cases.
\begin{Gev}\label{Gev1}
The number of possibilities for the parameter $x$ of a Cameron-Liebler $k$-set in $\pg(n,q)$, with $n \geq 2k+2$, equals $q^{k+1}\frac{q^{n-k}-1}{q^{k+1}-1} \gauss{n-k-1}{k}_q+2$.
\end{Gev}
\begin{proof}
Assume that $x\in [0,1]$. By Lemma \ref{Non-existence01}, $x=1$ or $x=0$. This makes two possibilities for $x$.

Now assume that $x>1$. By Lemma \ref{Prop}, $1<x\leq \frac{q^{n+1}-1}{q^{k+1}-1}$. Now we use Theorem \ref{non-integer} for every $t=2k+1$, thus we have that
$$x=1+\frac{C}{\gauss{n-k-1}{k}_q},$$
for $C \in \mathbb{N}$. Now we count the maximum number of possibilities for this integer $C$, this is, 
we count the number of $C\in \mathbb{N}$, such that
$$0< C \leq \left(\frac{q^{n+1}-1}{q^{k+1}-1}-1\right)\gauss{n-k-1}{k}_q.$$
This number equals $q^{k+1}\frac{q^{n-k}-1}{q^{k+1}-1} \gauss{n-k-1}{k}_q$. Together with the 2 possibilities for $x$ in the first case, we find the desired upper bound.
\end{proof}

\subsection{Gluing Cameron-Liebler sets together}
We already have seen that the intersection with a subspace and a Cameron-Liebler $k$-set in $\pg(n,q)$ or $\ag(n,q)$ defines a Cameron-Liebler $k$-set in that subspace. This subsection concerns the converse question. Can we glue Cameron-Liebler $k$-sets in subspaces together to obtain a Cameron-Liebler $k$-set in $\pg(n,q)$ or $\ag(n,q)$. The answer to that question is not immediately clear. We have found the following result.
\begin{St}
Let  $\mathcal{L}$ be a non-empty set of $k$-spaces in $\pg(n,q)$, with $n\geq 2k+2$. Let $2k+1\leq t\leq n-1$ and suppose  that $\mathcal{L}$ restricted to every $t$-dimensional subspace $\pi_i$ defines a Cameron-Liebler $k$-set of parameter $x_{\pi_i}$. Then the following properties are equivalent.
\begin{enumerate}
\item The set $\mathcal{L}$ is a Cameron-Liebler $k$-set of parameter $C$ in $\pg(n,q)$.
\item For every $k$-space $K$, the number $C:=\chi_\mathcal{L}(K)+\frac{\left(\sum_{K \in \pi_i}x_{\pi_i}-  \gauss{n-k}{t-k}_q\chi_\mathcal{L}(K) \right)} {\gauss{n-k-1}{t-k-1}_q }$ is a constant independent from $K$.
\end{enumerate}
The sum of the parameters $x_{\pi_i}$ runs over all $t$-spaces through  $K$.
\end{St}
\begin{proof}
Lemma \ref{Aid} proves that (1) implies (2). 

Now suppose that (2) holds. Choose an arbitrary $k$-space and calculate, similarly as in the proof of Lemma \ref{Aid}, 
the number of $k$-spaces of $\mathcal{L}$ skew to $K$. This number equals
$$|R|=\left(\sum_{K \in \pi_i}x_{\pi_i}-\gauss{n-k}{t-k}_q\chi_{\mathcal{L}}(K)\right)\frac{\gauss{t-k-1}{k}_q}{\gauss{n-2k-1}{t-2k-1}_q}q^{k^2+k}.$$
Note that if we fill in the constant $C$ from Property (2), we obtain the number
$$\left(C-\chi_{\mathcal{L}}(K)\right)\frac{\gauss{n-k-1}{t-k-1}_q\gauss{t-k-1}{k}_q}{\gauss{n-2k-1}{t-2k-1}_q}q^{k^2+k}.$$
This number simplifies to $\left( C-\chi_\mathcal{L}(K)\right) \gauss{n-k-1}{k}_q q^{k^2+k}$, which is the number from Corollary \ref{EquivalencenoParamP}, Property (2). So $\mathcal{L}$ is a Cameron-Liebler $k$-set of parameter $C$. 
\end{proof}

We can prove the following theorem  similar to the previous theorem. Its proof now relies on Lemma \ref{AidAff} instead of Lemma \ref{Aid}.
\begin{St}
Let  $\mathcal{L}$ be a non-empty set of lines in $\ag(n,q)$, with $n\geq 4$. Let $3 \leq t\leq n-1$ and 
suppose  that $\mathcal{L}$ restricted to every $t$-dimensional subspace $\pi_i$ defines a 
Cameron-Liebler line class of parameter $x_{\pi_i}$. Then the following properties are equivalent.
\begin{enumerate}
\item The set $\mathcal{L}$ is a Cameron-Liebler line class in $\ag(n,q)$ of parameter $C$.
\item For every line $\ell$, the number 
$C:=\frac{q^{t-2}-1}{q^{n-2}-1}\frac{\sum_{\ell \in \pi_i} x_{\pi_i}}{\gauss{n-3}{t-3}_q}-\frac{q^{n-1}-1}{q^{t-1}-1}\chi_\mathcal{L}(\ell)+\chi_\mathcal{L}(\ell)$, is a constant {\bf integer} independent from $\ell$. Also it holds that through every point at infinity we have exactly $C$ lines of $\cL$.
\end{enumerate}
\end{St}
\begin{proof}
From (1) to (2) follows in fact from Lemma \ref{AidAff}.

Now suppose that (2) holds. Choose an arbitrary line and calculate, as in the proof of Lemma~\ref{AidAff}, the number 
of lines of $\cL$ that are disjoint to a fixed line $\ell$ in the projective closure. This number equals
$$|R|=\frac{q^2}{\gauss{n-3}{t-3}_q}\frac{q^{t-2}-1}{q-1}\left( \sum_{\ell\in \pi_i}x_{\pi_i}-\gauss{n-1}{t-1}_q \chi_\mathcal{L}(\ell) \right).$$
Filling in our integer parameter, we obtain that this number equals
$$|R|=q^2\frac{q^{n-2}-1}{q-1}(C-\chi_\cL(\ell)).$$
Adding the $C-\chi_\cL(\ell)$ lines of $\cL$ through the point at infinity on $\ell$, we have the number in Theorem \ref{EquivalenceAff}. This proves the theorem.
\end{proof}

\section{Non-existence results}
In this section, we will prove some non-existence conditions on the parameters of Cameron-Liebler sets. We will use a similar strategy as
in \cite{DrudgeThesis}, and consider the Cameron-Liebler set induced in subspaces.
\begin{Opm}\cite[Lemma 3.1]{Jozefien}
Since the complement of a Cameron-Liebler $k$-set of parameter $x$ in $\pg(n,q)$ (or $\ag(n,q)$) is also a Cameron-Liebler $k$-set but of parameter $\frac{q^{n+1}-1}{q^{k+1}-1}-x$ (or $q^{n-k}-x$ respectively), we are only interested in those Cameron-Liebler $k$-sets with parameter at most $\frac12\frac{q^{n+1}-1}{q^{k+1}-1}$ in $\pg(n,q)$ or at most $\frac12q^{n-k}$ in $\ag(n,q)$.
\end{Opm}

\subsection{Cameron-Liebler $k$-sets in $\pg(n,q)$}

The main objective will be to prove the following theorem.

\begin{St}\label{NonEx}
Suppose that $n\geq 3k+3$ and $k\geq 1$. Let $\mathcal{L}$ be a Cameron-Liebler $k$-set of parameter $x$ in $\pg(n,q)$ such that $\mathcal{L}$ is not a point-pencil, nor the empty set. Then it holds that
$$x\geq\frac{q^{n-k}-1}{q^{2k+2}-1}+1.$$
\end{St}
\begin{proof}
Let $\mathcal{L}$ be a non-empty Cameron-Liebler $k$-set in $\pg(n,q)$ of parameter $x$. Choose a fixed $k$-set $K \in \mathcal{L}$.
Then, by Lemma~\ref{Aid}, for an arbitrary $t$, with  $3k+2 \leq t\leq n-1$, it holds that
\begin{equation}
x=\frac{\left( \sum_{K \in \pi_i}x_{\pi_i}-\gauss{n-k}{t-k}_q\right)}{\gauss{n-k-1}{t-k-1}_q}+1.
\end{equation}
As usual, $x_{\pi_i}$ denotes the parameter of the induced Cameron-Liebler $k$-set in the $t$-dimensional subspace $\pi_i$.
Now observe the following facts.
\begin{itemize}
\item Since $K\in \cL$ and, using Theorem \ref{Non-existence01}, we have that every Cameron-Liebler $k$-set in every $t$-dimensional space $\pi_i$ through $K$ has parameter $x_{\pi_i}\geq 1$;
\item if $\cL$ is not a point-pencil, by Theorem \ref{DrudgeGeneralPG}, $\cL$ is not a point-pencil in $\pi_i$. Hence, by Theorem \ref{CharX=1PG(n,q)}, every $x_{\pi_i}>1$;
\item if we assume that $t=3k+2$, then, by Theorem \ref{Non-existence01}, $x_{\pi_i}\geq 2$.
\end{itemize} 
Hence, under the assumptions of the theorem, we find that
$$x\geq \frac{\gauss{n-k}{t-k}_q}{\gauss{n-k-1}{t-k-1}_q}+1=\frac{\gauss{n-k}{2k+2}_q}{\gauss{n-k-1}{2k+1}_q}+1=\frac{q^{n-k}-1}{q^{2k+2}-1}+1,$$
where in the last two formulas, we let $t=3k+2$.
\end{proof}
\begin{Opm}
This result is also valid for Cameron-Liebler $k$-sets in $\ag(n,q)$ due to Theorem \ref{ToInftyAndBeond}. But for the affine case, we will improve these results.
\end{Opm}

\subsection{Cameron-Liebler $k$-sets in $\ag(n,q)$}
We use a similar strategy as in the previous subsection, but we will exploit Theorem \ref{ToInftyAndBeond}. 
\begin{Gev}\cite[Corollary 6.17]{Me2}\label{No2}
There do not exist Cameron-Liebler $k$-sets in $\ag(n,q)$ with parameter $x=2$, with $n\geq k+2$.
\end{Gev}
Note that the previous corollary does not always hold in $\pg(n,q)$. This would lead to the improvement of the previous section.
\begin{St}\label{NonExAG}
Suppose that $n\geq 2k+2$ and $k\geq 1$. Let $\mathcal{L}$ be a Cameron-Liebler $k$-set of parameter $x$ in $\ag(n,q)$ such that $\mathcal{L}$ is not a point-pencil, nor the empty set. Then 
$$x\geq 2\left(\frac{q^{n-k}-1}{q^{k+1}-1}\right)+1.$$
\end{St}
\begin{proof}
Consider the projective closure of $\ag(n,q)$. As usual, denote by $\pi_A$ the affine part of a projective subspace $\pi$. 
By Theorem \ref{ToInftyAndBeond}, every Cameron-Liebler $k$-set $\cL$ in $\ag(n,q)$ is a Cameron-Liebler $k$-set in $\pg(n,q)$. 
Suppose now that $\cL$ is not empty and choose an (affine) $k$-space $K\in \cL$. 
Choose $t$ arbitrary such that $2k+1\leq t\leq n-1$. By Lemma \ref{Aid}, 
\begin{equation}\label{eq:affine-x}
x=\frac{\left( \sum_{K \in \pi_i}x_{\pi_i}-\gauss{n-k}{t-k}_q\chi_{\mathcal{L}}(K)\right)}{\gauss{n-k-1}{t-k-1}_q}+\chi_{\mathcal{L}}(K),
\end{equation}
where $x_{\pi_i}$ denotes the parameter of the (projective) Cameron-Liebler $k$-set in the subspace $\pi_i$ of dimension $t$.

Now consider $\cL_{\pi_i}=\cL\cap [\pi_i]_k$ for an arbitrary subspace $\pi_i$. Then, by Theorem \ref{Subsp}, 
$\cL_{\pi_i}$ is a Cameron-Liebler $k$-set in  $\pi_i$. Since $\cL$ is a set of affine $k$-spaces, $\cL_{\pi_i}$ is fully contained in $\pi_{iA}$.
Hence, by Theorem \ref{ToInftyAndBeondConv}, $\cL_{\pi_i}$ is a Cameron-Liebler $k$-set in $\pi_{iA}$ of the same parameter $x_{\pi_i}$.  Note that from this follows that the parameter $x_{\pi_i}\in \mathbb{N}$. Since $\cL_{\pi_i}$ is not empty, we have that $x_{\pi_i} \neq 0$. And if there exists a subspace $\pi_i$ such that $x_{\pi_i}=1$, then, by Theorem \ref{Unique1AG}, it holds that $\cL\cap[\pi_{iA}]_k$ is a point-pencil and hence, by Theorem \ref{PPAG}, it holds that $\cL$ is a point-pencil. Thus we may assume that $x_{\pi_i}\not=1$. 

Furthermore, it holds that $x_{\pi_i}\geq 2$, and, by Corollary \ref{No2}, that $x_{\pi_i}\geq 3$. 
Using this lower bound on $x_{\pi_i}$ in Equation~\eqref{eq:affine-x}, we get
$$x\geq 2\left(\frac{\gauss{n-k}{t-k}_q}{\gauss{n-k-1}{t-k-1}_q}\right)+1.$$
Finally, using $t=2k+1$, we get
$$x\geq 2\left(\frac{\gauss{n-k}{k+1}_q}{\gauss{n-k-1}{k}_q}\right)+1=2\left(\frac{q^{n-k}-1}{q^{k+1}-1}\right)+1.$$
\end{proof}

\subsection{Improvement of existing results}
Up to our knowledge, the currently best known non-existence results for general $n, k$ and $q$ is the following theorems.
\begin{St}\cite[Theorem 4.9]{Jozefien}\label{JozefienClassific}
There are no Cameron-Liebler $k$-sets in $\pg(n,q)$, with $n\geq 3k+2$ and $q\geq 3$, of parameter $x$ if
	$$2\leq x\leq \frac{1}{\sqrt[8]{2}}q^{\frac{n}{2}-\frac{k^2}{4}-\frac{3k}{4}-\frac{3}{2}}(q-1)^{\frac{k^2}{4}-\frac{k}{4}+\frac{1}{2}}\sqrt{q^2+q+1}.$$
\end{St}
Note that this upper bound has size roughly $q^{\frac{n}{2}-k}$. 
\begin{St}\cite[Theorem 7]{Ihringer2021}\label{Ihringer}
Let $n\geq 2k+1$ and suppose that $\mathcal{L}$ is a Cameron-Liebler $k$-set in $\pg(n,q)$ of parameter $x$. If $16x\leq \min\{q^{\frac{n-k-l+2}{3}}, q^{\frac{n-2k-r}{3}}\}$, where $n+1=m(k+1)-r$ with $0\leq r< k+1$ and $\frac{q^{l-1}-1}{q-1}< x\leq \frac{q^l-1}{q-1}$.\\
Then $x\leq 2$ and $\mathcal{L}$ is trivial.
\end{St}
This bound would in the best case scenario have size roughly $\min\{q^{\frac{n-k+1}{3}}, q^{\frac{n-2k}{3}}\}=q^{\frac{n-2k}{3}}$.\\
In general no theorem of these above excludes the other, because both have different strengths in different cases of $(n,k)$.

For the affine case we have not found any non-existing conditions, 
except those that arose from non-existence conditions in $\pg(n,q)$. Hence, so far, in the affine case, the best known
non-existence conditions are induced by previous theorems.

\subsubsection{The improvement of Theorem \ref{NonEx}}
If we compare Theorem \ref{NonEx} with the other theorems above, it is clear that we require $n\geq 3k+3$ 
which is a slightly stronger condition, but we will show that our bound is better for large $n$. To do this we need to remark that the bound of Theorem \ref{NonEx} is roughly $q^{n-3k-2}$, which gives an improvement of the bound of 
Theorem \ref{JozefienClassific} and Theorem \ref{Ihringer}, when $n>4k+4$. Note that, due to Theorem \ref{ToInftyAndBeond}, this result is 
also valid for Cameron-Liebler $k$-sets in $\ag(n,q)$.

\subsubsection{The improvement of Theorem \ref{NonExAG}}
First of all we note that Theorem \ref{NonEx} is also valid for Cameron-Liebler $k$-sets in $\ag(n,q)$, with $n\geq 3k+3$. So we are left to compare if Theorem \ref{NonExAG} improves this theorem. We can see that the lower bound in Theorem \ref{NonExAG} is in fact roughly of order $q^{n-2k-1}$. Comparing this with the lower bound of Theorem \ref{NonEx} which is roughly $q^{n-3k-2}$, we can see a significant improvement.

If $n\geq3k+2$, we can compare Theorem \ref{NonExAG} with Theorem \ref{JozefienClassific}. We see that Theorem \ref{NonExAG} where the lower bound is roughly $q^{n-2k-1}$, which is in this case better than the bound in Theorem \ref{JozefienClassific} which is roughly $q^{\frac{n}{2}-k}$. Similar results hold comparing with Theorem \ref{Ihringer}.

If $n\geq 2k+2$, we compare with Theorem \ref{Ihringer}. In this case we also have a better bound.
%

\section{Conclusion}
The following step in the research can be to try to classify certain Cameron-Liebler $k$-sets in $\ag(n,q)$ or $\pg(n,q)$ for certain $k,n$ and $q$. \\

\section{Acknowledgments}
The authors want to thank Alexander Gavrilyuk and Ferdinand Ihringer for their suggestions which improved the article.

\section{Appendix}
This section gives an alternative proof of Theorem \ref{NonExAG} for $k=1$. This is in fact interesting because this proof is based on similar arguments as the proof of Theorem \ref{NonEx}. It is in our view nice to see that we can  obtain this result by arguments that are not based on exploiting the connection between Cameron-Liebler $k$-sets in $\ag(n,q)$ and $\pg(n,q)$.\\ Keep in mind that this proof only works for $k=1$, but if there is an equivalent of Theorem \ref{EquivalenceAff} for $k$-spaces, we could probably use the same technique as in Theorem \ref{NonEx}. Our guess is that by using this technique the result stays the same as Theorem \ref{NonExAG}. yet it would be interesting to see these arguments, and we could of course be wrong.

\subsection{For Cameron-Liebler line classes in $\ag(n,q)$}

\begin{St}\cite[Theorem 6.5 and 6.8]{Me2}\label{Classification}
Suppose that $\mathcal{L}$ is a Cameron-Liebler line class of parameter $x$ in $\ag(n,q)$, for $n\geq 3$, then the following statements are true.
\begin{itemize}
\item If $x=1$, then $\mathcal{L}$ is a point-pencil.
\item The parameter $x$ cannot be $2$.
\end{itemize}
\end{St}
Here we state a stronger result for Cameron-Liebler line classes in $\ag(n,q)$.
\begin{St}\label{ClasUpperBoundLines}
Suppose that $\mathcal{L}$ is a Cameron-Liebler line class of parameter $x$ in $\ag(n,q)$, with $n\geq 4$, then 
$x\in\{0, 1\}$ or $x\geq 2\left( \frac{q^{n-1}-1}{q^{2}-1}\right) +1.$
\end{St}
\begin{proof}
Suppose that $\mathcal{L}$ is a Cameron-Liebler line class in $\ag(n,q)$, which is not empty nor a point-pencil. 
Hence, by Theorem \ref{Classification}, $x>2$. Choose a line $\ell \in \mathcal{L}$. Then, by Lemma \ref{AidAff},
for $3\leq t \leq n-1$,
\begin{equation}\label{Paramx} x=\frac{q^{t-2}-1}{q^{n-2}-1}\frac{\sum_{\ell \in \pi_i} x_{\pi_i}}{\gauss{n-3}{t-3}_q}-\frac{q^{n-1}-1}{q^{t-1}-1}+1.
\end{equation}
We now have the following facts:
\begin{enumerate}
\item Since $\ell \in \mathcal{L}$, every Cameron-Liebler line class in every $t$-dimensional subspace $\pi_i$ has parameter $x_{\pi_i}\geq 1$.
\item If there exists a $t$-space where $\mathcal{L}\cap [\pi_i]_1$ has parameter $x_{\pi_i}=1$, then, by Theorem \ref{Classification}, it is a point-pencil. Hence, by Theorem \ref{PPAG}, it follows that $\mathcal{L}$ is a point-pencil. So we may suppose that $x_{\pi_i}>1$.
\item Using Theorem \ref{Classification}, we know that $x_{\pi_i}>2$.
\end{enumerate}
So, we conclude that for every $t$-space $\pi_i$ through $l\in \cL$, it holds that $x_{\pi_i}\geq 3$. So, using Equation (\ref{Paramx}), we obtain that
\begin{equation*}
\begin{split}
x&=\frac{q^{t-2}-1}{q^{n-2}-1}\frac{\sum_{\ell \in \pi_i} x_{\pi_i}}{\gauss{n-3}{t-3}_q}-\frac{q^{n-1}-1}{q^{t-1}-1}+1 \\
& \geq \frac{q^{t-2}-1}{q^{n-2}-1}\frac{3\gauss{n-1}{t-1}_q}{\gauss{n-3}{t-3}_q}-\frac{q^{n-1}-1}{q^{t-1}-1}+1 \\
& \geq 3 \left( \frac{q^{n-1}-1}{q^{t-1}-1} \right)-\frac{q^{n-1}-1}{q^{t-1}-1}+1\\
& \geq 2\left( \frac{q^{n-1}-1}{q^{t-1}-1}\right) +1\\ 
&\geq 2\left( \frac{q^{n-1}-1}{q^{2}-1}\right) +1,
\end{split}
\end{equation*}
where in the last line, we have chosen $t=3$.
\end{proof}

\bibliographystyle{plain}

\end{document}